\documentclass{amsart}
\usepackage[utf8]{inputenc}
\usepackage{amssymb,latexsym}
\usepackage{amsmath}
\usepackage{graphicx}
\usepackage{textcomp}
\usepackage[dvipsnames]{xcolor}

\usepackage{amsthm,amssymb,enumerate,graphicx, tikz}
\usepackage{amscd}
\usepackage{setspace}
\usepackage{comment}
\usepackage{hyperref}
\usepackage{cleveref}

\def\@logofont{\footnotesize}
\def\@setaddresses{\par
  \nobreak \begingroup
  \footnotesize
  \def\author##1{\nobreak\addvspace\bigskipamount}%
  \def\\{\par\nobreak}%
  \interlinepenalty\@M
  \def\address##1##2{\begingroup
    \par\addvspace\bigskipamount\indent
    \@ifnotempty{##1}{(\ignorespaces##1\unskip) }%
    {\scshape\ignorespaces##2}\par\endgroup}%
  \def\curraddr##1##2{\begingroup
    \@ifnotempty{##2}{\nobreak\indent\curraddrname
      \@ifnotempty{##1}{, \ignorespaces##1\unskip}\/:\space
      ##2\par}\endgroup}%
  \def\email##1##2{\begingroup
    \@ifnotempty{##2}{\nobreak\indent\emailaddrname
      \@ifnotempty{##1}{, \ignorespaces##1\unskip}\/:\space
      \ttfamily##2\par}\endgroup}%
  \def\urladdr##1##2{\begingroup
    \def~{\char`\~}%
    \@ifnotempty{##2}{\nobreak\indent\urladdrname
      \@ifnotempty{##1}{, \ignorespaces##1\unskip}\/:\space
      \ttfamily##2\par}\endgroup}%
  \addresses
  \endgroup
}
\renewcommand*\subjclass[2][2010]{%
  \def\@subjclass{#2}%
  \@ifundefined{subjclassname@#1}{%
    \ClassWarning{\@classname}{Unknown edition (#1) of Mathematics
      Subject Classification; using '2000'.}%
  }{%
    \@xp\let\@xp\subjclassname\csname subjclassname@#1\endcsname
  }%
}

\usepackage{graphicx,amsmath,amssymb}
\usepackage{algorithm}
\usepackage{mathdots, comment}
\usepackage[noend]{algpseudocode}

\newtheorem{theorem}{Theorem}[section]

\newtheorem*{theorem*}{Theorem}

\newtheorem{lemma}[theorem]{Lemma}

\theoremstyle{definition}

\theoremstyle{remark}
\newtheorem{remark}[theorem]{Remark}

\begin{document}
\title{Classifying Abelian Groups through Acyclic Matchings}


\author[M. Aliabadi and P. Taylor]{Mohsen Aliabadi$^{1}$ \and Peter Taylor$^2$}

\thanks{$^1$Department of Mathematics, University of California, San Diego, 
9500 Gilman Dr, La Jolla, CA 92093, USA. \url{maliabadisr@ucsd.edu}, ORCID: 0000-0001-5331-2540.\\
$^2$Independent Researcher, Valencia, Spain. \url{pjt33@cantab.net}, ORCID: 0000-0002-0556-5524.}

\begin{abstract}
The inquiry into identifying sets of monomials that can be eliminated from a generic homogeneous polynomial via a linear change of coordinates was initiated by E. K. Wakeford. This linear algebra problem prompted C. K. Fan and J. Losonczy to introduce the notion of acyclic matchings in the additive group $\mathbb{Z}^n$, subsequently extended to abelian groups by the latter author. Alon, Fan, Kleitman, and Losonczy established the acyclic matching property for $\mathbb{Z}^n$. This note aims to classify all abelian groups with respect to the acyclic matching property.
\end{abstract}

\maketitle

\section{Introduction}
Let \( (G,+) \) be an abelian group. Consider nonempty finite subsets \( A, B \subseteq G \). A {\it{matching}} from \( A \) to \( B \) is a bijection \( f : A \rightarrow B \) satisfying the condition \( a + f(a) \notin A \) for all \( a \in A \). If there exists a matching from $A$ to $B$, we say $A$ is {\it{matched}} to $B$. Obvious necessary conditions for the existence of a matching from \( A \) to \( B \) are \( |A| = |B| \) and \( 0 \notin B \). An abelian group \( G \) is said to possess the {\it{matching property}} if these conditions on \( A \) and \( B \) are sufficient to ensure the existence of a matching from \( A \) to \( B \). The question arises: which groups have the matching property? This question was addressed by Losonczy in \cite{Losonczy}:
\begin{theorem}\label{matching property}
    Let \( G \) be an abelian group. Then, \( G \) possesses the matching property if and only if \( G \) is torsion-free or cyclic of prime order.
\end{theorem}
For any matching $f:A\rightarrow B$, the associated multiplicity function $m_f : G \rightarrow \mathbb{Z}_{\geq 0}$ is defined via the rule:
\begin{align*}
  m_f(x) = |\{a \in A : a + f(a) = x\}|, 
\end{align*}
for all $x\in G$.

A matching $f : A \rightarrow B$ is called {\it{acyclic}} if for any matching $g: A \rightarrow B$, $m_f = m_g$ implies $f = g$. If there exists an acyclic matching from $A$ to $B$, we say $A$ is {\it{acyclically matched}} to $B$. A group $G$ has the {\it{acyclic matching property}} if for any pair of subsets $A$ and $B$ in $G$ with $|A| = |B|$ and $0 \notin B$, there is at least one acyclic matching from $A$ to $B$. Alon et al. proved in \cite{Alon} that the additive group \( \mathbb{Z}^n \) possesses the acyclic matching property. Losonczy extended this result to all abelian torsion-free groups by leveraging a total ordering compatible with the structure of abelian torsion-free groups, as established by Levi \cite{Levi}.
\begin{theorem}\label{a.m.p}
    Let $G$ be an abelian torsion-free group. Then G possesses
the acyclic matching property.
\end{theorem}
The concept of matching was introduced by Fan and Losonczy in \cite{Fan} as a tool to investigate Wakeford's classical problem regarding canonical forms for symmetric tensors \cite{Wakeford}. In particular, let $V$ be a $q$-dimensional complex vector space, $S^p(V)$ be the symmetric tensor of degree $p$ and $A$ be a fixed set  consisting of $q(q-1)$ monomials in $S^p(V)$, where $i_k > 0$ for all $k$. It is proved in \cite{Fan} that the generic element of $S^p(V)$ can be conjugated, via an appropriate linear transformation, to an element whose support lies outside of $A$. The main tool to prove this is by establishing the existence of an acyclic matching in a specific weighted bipartite graph. A linear analogue of matchings in vector subspaces over field extensions was introduced in \cite{Eliahou}. The matroidal counterpart of matchings was explored in \cite{Aliabadi 3}.\\

In this paper, our objective is to establish a necessary and sufficient condition for an abelian group to possess the acyclic matching property. In view of Theorems \ref{matching property} and \ref{a.m.p}, it naturally prompts an inquiry into the acyclic matching property in \( \mathbb{Z}/p\mathbb{Z} \). This problem was initially explored in \cite{Aliabadi 1}, where it was proved that infinitely many primes \( p \) exist for which \( \mathbb{Z}/p\mathbb{Z} \) lacks the acyclic matching property. Particularly, it was shown that for the following two families of primes, \( \mathbb{Z}/p\mathbb{Z} \) does not exhibit the acyclic matching property:
\begin{itemize}
    \item \( p \equiv -1 \pmod{8} \) 
    \item The multiplicative order of $2$ modulo \( p \) is an odd number.
\end{itemize}
Considering the latter condition along with a result from \cite{Hasse} concerning the parity of the multiplicative order of \(2\) modulo \(p\), it is observed that primes \(p\) where \(\mathbb{Z}/p\mathbb{Z}\) lacks the acyclic matching property have a density of at least \(\frac{7}{24}\). Additionally, employing a rectification principle \cite{Lev}, which posits that a suitably small subset of \(\mathbb{Z}/p\mathbb{Z}\) can be embedded in integers while preserving certain additive properties, it is proved in \cite{Aliabadi 0} that for $A,B\subset \mathbb{Z}/p\mathbb{Z}$, $A$ is acyclically matched to $B$ provided that $0\notin B$ and $|A|=|B|\leq\sqrt{\log_2p}-1$.

\section{Main results}

In this section, we present our main results concerning the classification of all abelian groups with respect to the acyclic matching property. In particular, as outlined in Theorem \ref{acyclic-cyclic} and Theorem \ref{nes-suf}, we establish initially that for $p>5$, $\mathbb{Z}/p\mathbb{Z}$ lacks the acyclic matching property. Furthermore, we determine a necessary and sufficient condition for an abelian group to possess the acyclic matching property. We initiate our investigation by counting the matchings with the same multiplicity functions in $\mathbb{Z}/n\mathbb{Z}$.

\begin{lemma}\label{construction}
Let $m$ and $n$ be two integers with $m>1$ and $n \ge m + 4$.  Let $A=\mathbb{Z}/n\mathbb{Z} \setminus \{0, 1, 3\}$ and $B= \mathbb{Z}/n\mathbb{Z}\setminus \{0, 1, m\}$. Given a matching $f:A\rightarrow B$, we encode the multiplicity function $m_f$ as $c_0^{w_0} c_1^{w_1} c_3^{w_3}$ where the exponent of $c_k$ counts values $a \in A$ for which $a + f(a) = k$. Then the coefficients of the generating function $$\begin{pmatrix} c_1^2 c_3 & c_0 c_3^2 & c_1 c_3^2 \end{pmatrix} \begin{pmatrix} 0 & c_3 & 0 \\ c_1 & 0 & c_3 \\ c_0 & 0 & 0 \end{pmatrix}^{n-m-4} \begin{pmatrix} 1 & 0 & 0 \\ 0 & 1 & 0 \\ 0 & 0 & 0 \end{pmatrix} \begin{pmatrix} c_1 & 0 & c_3 \\ c_0 & 0 & 0 \\ 0 & c_0 & 0 \end{pmatrix}^{m-2} \begin{pmatrix} 1 \\ 0 \\ 0 \end{pmatrix}$$ count the matchings from $A$ to $B$ with the corresponding multiplicity function.\footnote{In fact this holds for $n > \max(m,3)$. Nevertheless, we refrain from addressing the additional cases to maintain simplicity in the proof.}
\end{lemma}
\begin{proof}
We match the elements of $B$ in ascending order. Each $b \in B$ can be matched with $\{-b, 1-b, 3-b\} \setminus \{0, 1, 3\}$. It is impossible that $-b$ has been matched with any $b' < b$, so the state we need to track when we reach $b$ is which of $\{1-b, 2-b, 3-b\}$ are unmatched. For $b < m$, exactly one of them is in $A$ and unmatched, so we have three possible states $s_k$ where we have not yet matched $k-b$. We begin the matching at $b=2$ in state $s_1$ because $2-b = 0$ and $3-b = 1$ are not in $A$. For elements $\{2, \ldots, m-1\} \subseteq B$ we apply transitions $s_1 \stackrel{c_1}\rightarrow s_1$, $s_1 \stackrel{c_0}\rightarrow s_2$, $s_2 \stackrel{c_0}\rightarrow s_3$, $s_3 \stackrel{c_3}\rightarrow s_1$.

Then we do not match $m$, so we have a change in states to always have two of $\{1-b, 2-b, 3-b\}$ unmatched; $s_1 \stackrel 1\rightarrow s_{12}$, $s_2 \stackrel 1\rightarrow s_{13}$. Note that there is no transition from $s_3$ because if $3-m$ is unmatched at this point it can never be matched.

Continuing for $b = m + 1$ to $b = n-4$ inclusive we have transitions $s_{12} \stackrel{c_0}\rightarrow s_{23}$, $s_{12} \stackrel{c_1}\rightarrow s_{13}$, $s_{13} \stackrel{c_3}\rightarrow s_{12}$, $s_{23} \stackrel{c_3}\rightarrow s_{13}$.

Finally we match $\{n-3, n-2, n-1\}$. From state $s_{12}$ we have to match as $f(4) = n-3$, $f(5) = n-2$, $f(2) = n-1$ with weight $c_1^2 c_3$. From state $s_{13}$ we have to match $f(6) = n-3$, $f(4) = n-1$, $f(2) = n-2$ with weight $c_0 c_3^2$. And from state $s_{23}$ we have to match $f(6) = n-3$, $f(5) = n-2$, $f(2) = n-1$ with weight $c_1 c_3^2$.
\end{proof}
\begin{remark}
    Notice that the characteristic polynomial of $\begin{pmatrix} 0 & c_3 & 0 \\ c_1 & 0 & c_3 \\ c_0 & 0 & 0 \end{pmatrix}$ is ${x^3 - c_1 c_3 x - c_0 c_3^2}$.
\end{remark}
 The following lemma will be employed to prove explicit expressions for specific $m$ by induction.

\begin{lemma}\label{binomial-recurrence} For $d, e \in \mathbb{Z}$ and $m \in \mathbb{N}$, $$F_n = \sum_{\substack{w_0 + w_1 + w_3 = n-3 \\ 2w_0 + w_1 + 1 = w_3+m}} \binom{w_0+w_1-d}{w_1-e} c_0^{w_0} c_1^{w_1} c_3^{w_3}$$ satisfies a linear recurrence with characteristic polynomial ${x^3 - c_1 c_3 x - c_0 c_3^2}$.
\end{lemma}
\begin{proof}
Using square brackets enclosing a monomial (e.g. $[c_0^{w_0} c_1^{w_1} c_3^{w_3}]$) to denote the coefficient extraction operator and square brackets enclosing a boolean expression (e.g. $[w_0 + w_1 + w_3 = n-3]$) to denote the Iverson bracket which evaluates to $1$ when the expression is true and $0$ otherwise, we have
\begin{eqnarray*}
&& [c_0^{w_0} c_1^{w_1} c_3^{w_3}] c_1 c_3 \sum_{\substack{u_0 + u_1 + u_3 = n-5 \\ 2u_0 + u_1 + 1 = u_3+m}} c_0^{u_0} c_1^{u_1} c_3^{u_3} \binom{u_0+u_1-d}{u_1-e} + \\
&& [c_0^{w_0} c_1^{w_1} c_3^{w_3}] c_0 c_3^2 \sum_{\substack{v_0 + v_1 + v_3 = n-6 \\ 2v_0 + v_1 + 1 = v_3+m}} c_0^{v_0} c_1^{v_1} c_3^{v_3} \binom{v_0+v_1-d}{v_1-e} \\
&=& [w_0 + w_1 + w_3 = n-3][2w_0+w_1+1=w_3+m] \binom{w_0+w_1-d-1}{w_1-e-1} + \\
&& [w_0 + w_1 + w_3 = n-3][2w_0 + w_1 + 1 = w_3+m] \binom{w_0+w_1-d-1}{w_1-e} \\
&=& [w_0 + w_1 + w_3 = n-3][2w_0+w_1+1=w_3+m] \binom{w_0+w_1-d}{w_1-e}.
\end{eqnarray*}
\end{proof}

\begin{lemma}\label{case-m=2}
    If we take $m=2$ in the construction of Lemma \ref{construction} then for integer $n \ge 6$ the generating function for the number of matchings in $\mathbb{Z}/n\mathbb{Z}$ is $$\sum_{\substack{w_0 + w_1 + w_3 = n-3 \\ 2w_0 + w_1 + 1 = w_3+2}} \binom{w_0+w_1}{w_1} c_0^{w_0} c_1^{w_1} c_3^{w_3}.$$
\end{lemma}
\begin{proof}
	The initial terms $c_1^2 c_3$ for $n=6$, $2c_0 c_1 c_3^2$ for $n=7$, and $c_1^3 c_3^2 + c_0^2 c_3^3$ for $n=8$ can be manually verified; and then Lemma \ref{binomial-recurrence} completes a proof by induction.
\end{proof}

\begin{lemma}\label{case-m=6}
    If we take $m=6$ in the construction of Lemma \ref{construction} then for integer $n \ge 10$ the generating function for the number of matchings in $\mathbb{Z}/n\mathbb{Z}$ is $$\sum_{\substack{w_0 + w_1 + w_3 = n-3 \\ 2w_0 + w_1 + 1 = w_3+6}} \left( \binom{w_0+w_1-2}{w_1} + \binom{w_0+w_1-3}{w_1-1} + \binom{w_0+w_1-3}{w_1-3} \right) c_0^{w_0} c_1^{w_1} c_3^{w_3}.$$
\end{lemma}
\begin{proof}
	Again, the initial terms can be manually verified and Lemma \ref{binomial-recurrence} completes a proof by induction.
\end{proof}

\begin{lemma}\label{general case}
    If $n>5$ is coprime to $6$, then $\mathbb{Z}/n\mathbb{Z}$ does not possess the acyclic matching property.
\end{lemma}
\begin{proof}
	We assume that $A$ and $B$ are as in Lemma \ref{construction}. We prove that $A$ cannot be acyclically matched to $B$. We split into two cases:
    \medskip

{\bf Case 1}.
    If $n = 6k+1$ then we take $m=2$ in the construction of Lemma \ref{construction}. By Lemma \ref{case-m=2} we have generating function:$$\sum_{\substack{w_0 + w_1 + w_3 = 6k-2 \\ 2w_0 + w_1 = w_3+1}} \binom{w_0+w_1}{w_1} c_0^{w_0} c_1^{w_1} c_3^{w_3}.$$\\
    For an acyclic matching to occur, we require a binomial coefficient which is equal to $1$, so either $w_0 = 0$ or $w_1 = 0$. But the linear relationships $w_0 + w_1 + w_3 = 6k-2$ and $2w_0 + w_1 = w_3+1$ combine to yield $3w_0 + 2w_1 = 6k-1$, so modular considerations rule out both possibilities.
\medskip

{\bf Case 2}.
    If $n = 6k+5$ we proceed similarly with $m=6$ and Lemma \ref{case-m=6}. This time the coefficients are the sum of three binomial coefficients. However, none of those binomial coefficients can be $1$ for integer exponents $w_0, w_1, w_3$ which satisfy the constraints of the sum, and since they are non-negative their sum can likewise never be $1$.
\end{proof}
\begin{remark}
    In Lemma \ref{general case} we come up with two sets $A$ and $B$ of size $n-3$ that are not acyclically matched. However, Example 2.2 in \cite{Aliabadi 0} demonstrates that in an abelian group $G$ if $|A|=|B|=|G|-1$ or $|A|=|B|=|G|-2$, $A$ must be acyclically matched with $B$. Therefore, $A$ and $B$ in Lemma \ref{general case} represent the largest possible sets that fail to exhibit acyclic matchings.
\end{remark}
In the following theorem, we classify all cyclic groups of prime order with respect to the acyclic matching property. This lemma gives a negative answer to Question 2.4 in \cite{Aliabadi 1}, and serves as the ``{\it{finite cyclic group}}''
 analogue of Theorem 1 in \cite{Alon} and Theorem 4.1 in \cite{Losonczy}.

\begin{theorem}\label{acyclic-cyclic}
    If \( p > 5 \) is prime, then \( \mathbb{Z}/p\mathbb{Z} \) does not possess the acyclic matching property.
\end{theorem}
\begin{proof}
The proof is immediate from Lemma \ref{general case}.
\end{proof}

\begin{remark}
    It is worth mentioning that Lemma \ref{general case} and Theorem \ref{acyclic-cyclic} are applicable in a more general setting. In essence, we can assert that for $n > 5$, $\mathbb{Z}/n\mathbb{Z}$ lacks the acyclic matching property. The case where $n$ is prime follows directly from Theorem \ref{acyclic-cyclic}. For non-prime $n$, choose $a$ such that $\gcd(a, n) \neq 1$. Choose $x \in \mathbb{Z}/n\mathbb{Z} \setminus \langle a \rangle$ and set $A := \langle a \rangle$ and $B := (\langle a \rangle \cup \{x\}) \setminus \{0\}$. Then $A$ is not matched to $B$, let alone acyclically matched. Therefore, $\mathbb{Z}/n\mathbb{Z}$ does not possess the acyclic matching property. With all this, we are primarily interested in the scenario where $n$ is prime, in view of Theorem \ref{matching property}.
\end{remark}

We are now ready to provide a necessary and sufficient condition for abelian groups to possess the acyclic matching property. This theorem serves as the ``{\it{acyclic matching property}}''
 analogue of Theorem \ref{matching property}.
\begin{theorem}\label{nes-suf}
    An abelian group \( G \) possesses the acyclic matching property if and only if either \( G \) is torsion-free or \( G = \mathbb{Z}/p\mathbb{Z} \), where \( p \in \{2,3,5\} \).
\end{theorem}
\begin{proof}
    
    Suppose \( G \) possesses the acyclic matching property. Then it also satisfies the matching property. From Theorem \ref{matching property}, it follows that \( G \) is either torsion-free or cyclic of prime order. However, according to Lemma \ref{acyclic-cyclic}, if \( G \) is cyclic of prime order with acyclic matching property, then \( p\leq 5 \), as claimed. \\
    Conversely, it follows from Theorem \ref{a.m.p} that torsion-free groups possess the acyclic matching property. Also, the simulation results in \cite[page 153, Algorithm 3]{Aliabadi 2} demonstrate that \( \mathbb{Z}/p\mathbb{Z} \) for \( p = 2,3,5 \) exhibits the acyclic matching property, thereby completing the proof.
\end{proof}
 
\textbf{Data sharing:} Data sharing not applicable to this article as no datasets were generated or analysed.\\
\textbf{Conflict of interest:} To our best knowledge, no conflict of interests, whether of financial or personal nature, has influenced the work presented in this article.\\
\textbf{Acknowledgement:} We wish to thank two anonymous referees for their careful reading of the manuscript, their comments, corrections and suggestions for improvement.

\end{document}